\documentclass[12pt]{amsart}

\usepackage{MnSymbol}
\newtheorem{theorem}{Theorem}[section]
\newtheorem{lemma}[theorem]{Lemma}
\newtheorem{prop}[theorem]{Proposition}
\newtheorem{cor}[theorem]{Corollary}

\newtheorem{definition}[theorem]{Definition}

\theoremstyle{remark}

\newcommand{\rr}{{\mathbb R}}

\newcommand{\N}{{\mathbb N}}

\newcommand{\dF}{\mathbb{F}}
\newcommand{\dE}{\mathbb{E}}

\newcommand{\wt}{\widetilde}

\usepackage{comment}
\usepackage{tikz}
\usepackage{float}
\usepackage{relsize}
\usepackage{enumitem}
\usepackage[T1]{fontenc}
\usepackage[english]{babel}
\usepackage{color}
\usepackage{graphicx}
\usepackage{amsmath}

\date{\today}

\begin{document}

\title[Regularization of differential equations by two fractional noises]{Regularization of differential equations by two fractional noises} 

\author{David Nualart}\thanks{David Nualart was supported by the NSF grant  DMS 1811181}
\address{David Nualart, Department of Mathematics, University of Kansas, 405 Snow Hall, Lawrence, Kansas, 66045, USA}
\email{nualart\@@{}ku.edu}

\author{Ercan S\"onmez}
\address{Ercan S\"onmez, Department of Statistics, University of Klagenfurt, Universit\"atsstra{\ss}e 65--67, 9020 Klagenfurt, Austria}
\email{ercan.soenmez\@@{}aau.at}

\begin{abstract}
In this paper we show the existence and uniqueness of a solution for a stochastic differential equation driven by an additive noise which is the sum of two fractional Brownian motions with different Hurst parameters. The proofs are based on the techniques of fractional calculus and Girsanov theorem. In particular, we show that the regularization effect of the fractional Brownian motion with the smaller Hurst index
dominates.
\end{abstract}

\maketitle

\allowdisplaybreaks

\section{Introduction}

The aim of this paper is to show the existence and uniqueness of a strong solution for the stochastic differential equation
\[
X_t= x_0 + \int_0^t b(s,X_s) ds + B^{H_1}_t  +B^{H_2}_t , \quad t\in [0,T],
\]
where $x_0\in \mathbb{R}$, $B^{H_1}$ and $B^{H_2}$ are two fractional Brownian motions with  different Hurst parameters $H_1 \in (0,1)$ and $ H_2\in(0,1) $, respectively.
Our main assumption is that both fractional Brownian motions can be expressed as Volterra processes with respect to the same standard Brownian motion $W$, that is, 
\begin{align*}
B_t^{H_1}  = \int_0^t K_{H_1}(t,s) dW_s \quad \text{ and } \quad
B_t^{H_2}  = \int_0^t K_{H_2}(t,s) dW_s
\end{align*}
for suitable square integrable kernels $K_{H_i}(t,s) $, $i=1,2$.   We will impose  weak regularity assumptions on the drift $b(t,x)$. More precisely,
if one of the Hurst parameters is in $(0, \frac 12]$, then it is enough to assume that $b(t,x)$ is  Borel measurable and it has linear growth in $x$, whereas if both Hurst parameters are strictly larger than $\frac 12$, then  we will assume that $b$ is  $\beta$-H\"older continuous in   the time variable and $\gamma$-H\"older continuous in the space variable, with $\beta > 1- \frac{1}{2\min(H_1,H_2)}$ and $\gamma > \min(H_1,H_2) -\frac 12$ (see conditions (A1) and (A2) below).

Originally motivated by applications in finance there has been increasing interest in the mixed fractional Brownian motion, which is the sum of a standard Brownian motion and a fractional Brownian motion. Such a process has first been studied in \cite{cheridito}. Classical results for stochastic differential equations with respect to mixed fractional Brownian motions can be found in \cite{gu, KU2, MP, Mish, Mish2} and results for such stochastic differential equations with irregular drift coefficient in \cite{s}. The noise involved in the stochastic differential equations considered in this paper can be seen as a generalization of a process which has quite recently been studied in \cite{dufit} and was also discussed in \cite{Mishbook}. In this setting we allow for general Hurst indexes and don't restrict to, but include the case in which one of the noises is a standard Brownian motion. Such processes have also been investigated in \cite{miao, el, thale, zahle, zili} in a different framework. Moreover, unlike in the just mentioned references in the present paper the two fractional Brownian motions are not independent.

Since the pioneering works by  Zvonkin \cite{z} and Veretennikov \cite{v}, there has been  a lot of interest in the problem of regularization
of stochastic differential equations by noise.  In the case of the fractional  Brownian motion, the regularization effect  was first established by Nualart and Ouknine in \cite{NO}.

The proof of our results follow the methodology introduced by Nualart and Ouknine in \cite{NO} for the case of an additive fractional Brownian motion. The main ingredient of the proof is the Girsanov theorem for the fractional Brownian motion, established by Decreusefond and \"Ust\"unel in \cite{DU}. A key challenge in this paper is to extend and apply the Girsanov theorem for two noises given by the sum of two (dependent) fractional Brownian motions by using profound techniques of fractional operator theory.

\section{Preliminaries}
\subsection{Fractional calculus}
In this subsection we review  some basic definition and properties of the fractional operators. For more information on this topic we refer to 
\cite{Sam}.
 
 If $f \colon [a,b] \rightarrow \rr$ is an integrable function and $\alpha \in (0, \infty)$, the left
fractional Riemann-Liouville integral of $f$ of order $\alpha $ on $%
(a,b)$ is given at almost all $x$ by
\[
I_{a+}^{\alpha }f(x)=\frac{1}{\Gamma (\alpha )}\int_{a}^{x}(x-y)^{\alpha
-1}f(y)dy,
\]
where $\Gamma $ denotes the Euler function.
 We denote by $I_{a+}^{\alpha }(L^{p})$ the image
of $L^{p}([a,b])$ by the operator $I_{a+}^{\alpha }$ . If $f\in
I_{a+}^{\alpha }(L^{p}),$ the function $\phi $ such that $%
f=I_{a+}^{\alpha }\phi $ is unique in $L^{p}$ and it agrees with the
left-sided \ Riemann-Liouville derivative of $f$ of order $%
\alpha $ defined by
\[
D_{a+}^{\alpha }f(x)=\frac{1}{\Gamma (1-\alpha )}\frac{d}{dx}\int_{a}^{x}%
\frac{f(y)}{(x-y)^{\alpha }}dy.
\]
The derivative of $f$ has the following Weil representation:
\begin{equation} 
D_{a+}^{\alpha }f(x)=\frac{1}{\Gamma (1-\alpha )}\left( \frac{f(x)}{%
(x-a)^{\alpha }}+\alpha \int_{a}^{x}\frac{f(x)-f(y)}{(x-y)^{\alpha +1}}%
dy\right) \mathbf{1}_{(a,b)}(x),  \label{we}
\end{equation}
where the convergence of the integrals at the singularity $x=y$ holds in $%
L^{p}$-sense.

When $\alpha p>1$ any function in $I_{a+}^{\alpha }(L^{p})$ is $\left(
\alpha -\frac{1}{p}\right) $-H\"{o}lder continuous. On the other hand, any
H\"{o}lder continuous function of order $\beta >\alpha $ has fractional
derivative of order $\alpha $. That is, $C^{\beta }([a,b])\subset
I_{a+}^{\alpha }(L^{p})$ for all $p>1$.

We have the first composition formula
\[
I_{a^{+}}^{\alpha }(I_{a^{+}}^{\beta }f)=I_{a^{+}}^{\alpha +\beta }f.
\]
By definition,  for $f\in I_{a+}^{\alpha }(L^{p}),$%
\[
I_{a+}^{\alpha }(D_{a+}^{\alpha }f)=f
\]
and for a general $f\in L^{1}([a,b])$ we have
\[
D_{a+}^{\alpha }(I_{a+}^{\alpha }f)=f.
\]
If $f\in I_{a+}^{\alpha +\beta }$ $(L^{1})$, $\alpha \geq 0,$ $\beta \geq
0,$ $\alpha +\beta \leq 1$ we have the second composition formula
\[
D_{a+}^{\alpha }(D_{a+}^{\beta }f)=D_{a+}^{\alpha +\beta }f.
\]

\subsection{Fractional Brownian motion}

The  fractional Brownian motion
with Hurst parameter $0<H<1$, denoted by
 $B^{H}=(B_{t}^{H})_{t\in [0,T]}$  is a zero-mean  Gaussian process with covariance function
 \[
 R_H(t,s):= t^{2H} + s^{2H} - |t-s|^{2H}.
 \]
 We assume that $B^H$ is defined 
 on a probability space $(\Omega ,%
\mathbb{F},P)$.  

Consider the isomorphism
\[
 K_H \colon L^2 ([0,T]) \to I_{0+}^{H+\frac{1}{2}}  ( L^2 ([0,T]) 
 \]
given by 
\begin{align*}
(K_H h) (s) & = I_{0+}^{2H} s^{\frac{1}{2} - H} I_{0+}^{\frac{1}{2} - H} s^{H-\frac{1}{2} } h, \quad H \leq \frac12, \\
(K_H h) (s) & = I_{0+}^{1} s^{H-\frac{1}{2} } I_{0+}^{H-\frac{1}{2} } s^{\frac{1}{2} -H} h, \quad H > \frac12.
\end{align*}
for all $h \in L^2 ([0,T])$ and $s \in [0,T]$.

Its inverse operator $K_H^{-1}$ satisfies for $s \in [0,T]$ that
\begin{align} \label{pG}
\begin{split}
K_H^{-1} h &= s^{H-\frac{1}{2}} I_{0+}^{\frac{1}{2} - H} s^{\frac{1}{2} - H} h',\quad H \leq \frac12, \\
K_H^{-1} h &= s^{H-\frac{1}{2}} D_{0+}^{H-\frac{1}{2}} s^{\frac{1}{2} - H} h', \quad H > \frac12.
\end{split}
\end{align}
where $h \in I_{0+}^{H+\frac{1}{2}}  ( L^2 ([0,T]) )$ is absolutely continuous with respect to the Lebesgue measure.

Then, the covariance kernel $R_{H}(t,s)$ can be written as
\[
R_{H}(t,s)=\int_{0}^{t\wedge s}K_{H}(t,r)K_{H}(s,r)dr,
\]
where $K_{H}$ is a square integrable kernel given by $ K_H(t,s) := (K^*_H \mathbf{1}_{[0,t]}) (s)$
and $K^*_H$ denotes the adjoint operator of $K_H$.
  Moreover, there exists a standard Brownian motion $W=(W_t)_{t \in [0, T]}$ such that
\begin{equation} \label{ecu1}
B_t^H = \int_0^t K_H(t,s) dW_s ,\quad t \in [0,T].
\end{equation}

\subsection{Girsanov theorems}

 Girsanov theorem will play a fundamental role in the proof of our results. First, we will state the classical Girsanov theorem for the Brownian motion.

 Let $(\dF_{t} )_{t\in [0,T]}$   be a
right-continuous increasing family of $\sigma $-fields on $(\Omega ,\mathbb{F},P)$ such that $\dF_{0}$ contains the sets of probability zero. 
 We recall that  an $\dF$-Brownian motion $W=(W_t)_{t\in [0,T]}$ is a Brownian motion,  such that $W_t$ is $\dF_t$-measurable for each $t\in [0,T]$, implying that $W$ is $\dF$-adapted,  and $W_t-W_s$ is independent of $\dF_s$ for all $0\le s\le t \le T$.
 
 We say that a fractional Brownian motion $B^H$ adapted to the filtration $ \dF$ is an $ \dF$-fractional Brownian motion, if the the process $W$ given by  \eqref{ecu1} is an $\dF$-Brownian motion.

\begin{theorem} 
  Let $(\psi_t)_{t \in [0, T]}$ be an  $\dF$-adapted stochastic process such that
$\mathbb{E} \left[ \int_0^t \psi_s^2 ds \right] < \infty$
for every $t\in [0,T]$. If
\[
L_T:= \exp \left( \int_0^T\psi_s dW_s - \frac12 \int_0^T \psi_s^2 ds\right)
 \]
 satisfies $\dE[L_T]=1$, 
then the process $(\wt{W}_t)_{t \in [0, T]}$ with
\[
 \wt{W}_t = W_t - \int_0^t \psi_s ds,\quad  t\in (0,T],
 \]
is an $\dF$-Brownian motion with respect to the probability measure $Q$ given by
$\frac{dQ}{dP} = L_T$.
 \end{theorem}

As a consequence (see  \cite[Theorem 4.9]{DU}), one can show the following version of Girsanov theorem for the fractional Brownian motion. 

\begin{theorem} \label{thm1}
Let $u=(u_t)_{t \in [0, T]}$ be an $\dF$-adapted stochastic process with integrable paths and suppose that  $\int_0^\bullet u_s ds \in  I_{0+}^{H+\frac{1}{2}} ( L^2 ([0,T]) )$ almost surely. Set
\[
\psi_t=K_H^{-1} \left(\int_0^\bullet u_s ds\right)(t)
\]
and consider the shifted process
\begin{equation*}
\wt{B}_t^H = B_t^H - \int_0^t u_s ds, \quad t \in [0,T].
\end{equation*}
If
\[
L_T:= \exp \left( \int_0^T\psi_s dW_s - \frac12 \int_0^T \psi_s^2 ds\right)
 \]
 satisfies $\dE[L_T]=1$, 
then the shifted process $\wt{B}^H=(\wt{B}^H_t)_{t \in [0, T]}$ is an $\dF$-fractional Brownian motion with Hurst parameter $H$ under the probability measure $Q$ defined by $\frac{dQ}{dP}= L_T$.
\end{theorem}

Notice that in Theorem \ref{thm1},  by \eqref{pG}   we have
\begin{align*}
	\psi_t & = t^{H-\frac{1}{2}} I_{0+}^{\frac{1}{2} - H} t^{\frac{1}{2} - H} u_t, \quad \text{ if } H <\frac12, \\
	\psi_t & = t^{H-\frac{1}{2}} D_{0+}^{H-\frac{1}{2} } t^{\frac{1}{2} - H} u_t, \quad \text{ if } H >\frac12.
	\end{align*}
	for every $t \in (0,T]$.  If $H=\frac 12$, then $\psi_t=u_t$ for all $t\in [0,T]$. 
	Moreover, under $Q$, 
\[
 \wt{W}_t = W_t  -\int_0^t \psi_s ds 
 \]
 is an $\dF$-Brownian motion and
 the shifted process $(\wt{B}^H_t)_{t \in [0, T]}$ can be expressed as
	\begin{align*}
	\wt{B}^H_t & = \int_0^t K_H(t,s) dW_s - \int_0^t u_s ds \\
	& = \int_0^t K_H(t,s) d{W}_s- (K_H\psi)_t\\
	&= \int_0^t K_H(t,s) d{W}_s-  \int_0^t K_H(t,s) \psi_s ds \\
	&=\int_0^t K_H(t,s) d\wt{W}_s.
	\end{align*}

We state the following Girsanov theorem, whose proof is similar to  that of Theorem \ref{thm1} above.
\begin{theorem}\label{GT2}
	Let $(u_t)_{t \in [0, T]}$, $(v_t)_{t \in [0, T]}$ be $\dF$-adapted processes with integrable paths 
	such that  $\int_0^\bullet u_s ds \in  I_{0+}^{H_1+\frac{1}{2}} ( L^2 ([0,T]) )$ and $\int_0^\bullet v_s ds \in  I_{0+}^{H_2+\frac{1}{2}} ( L^2 ([0,T]) )$ 	 almost surely. Suppose that
	\[
	\psi_t:= \left( K_{H_1}^{-1} \int_0^\cdot u_r dr \right) (t)=\left( K_{H_2}^{-1} \int_0^\cdot v_r dr \right) (t),
	 \]
	for every $t \in (0,T]$.
	If
\[
L_T:= \exp \left( \int_0^T\psi_s dW_s - \frac12 \int_0^T \psi_s^2 ds\right)
 \]
 satisfies $\dE[L_T]=1$, then, under the probability measure $Q$ defined by $\frac{dQ}{dP} = L_T$, the processes $(\wt{B}^{H_1}_t)_{t \in [0, T]}$, $(\wt{B}^{H_2}_t)_{t \in [0, T]}$ with
	\[
	 \wt{B}^{H_1}_t = {B}^{H_1}_t  -\int_0^t u_s ds \quad \text{and} \quad \wt{B}_t^{H_2} = B_t^{H_2} -\int_0^t v_s ds, \quad t \in [0,T],
	 \]
	are $\dF$-fractional Brownian motions with Hurst indices $H_1$ and $H_2$, respectively. Moreover, there exists a standard Brownian motion $(\wt{W}_t)_{t \in [0, T]}$ with respect to $Q$ such that for every $t \in [0,T]$ it holds
	\begin{align*}
	\wt{B}_t^{H_1} = \int_0^t K_{H_1}(t,s) d\wt{W}_s \quad \text{ and } \quad \wt{B}_t^{H_2} = \int_0^t K_{H_2}(t,s) d\wt{W}_s.
	\end{align*}
\end{theorem}

\section{Existence of a weak solution}

Consider the stochstic differential equation
\begin{align}\label{one2}
X_t = x_0+ \int_0^t b(s,X_s) ds + B^{H_1}_t + B_t^{H_2},  \quad t\in [0,T],
\end{align}
where $B^{H_1}=(B^{H_1}_t)_{t \in [0, T]}$ and $B^{H_2}=(B^{H_2}_t)_{t \in [0, T]}$  are two fractional Brownian motions with different Hurst parameters $H_1, H_2 \in (0,1)$ and $x_0\in \mathbb{R}$. We assume that there is a Brownian motion $W=(W_t)_{t\in 0,T]}$, such that 
\begin{align*}
B_t^{H_1}  = \int_0^t K_{H_1}(t,s) dW_s \quad \text{ and } \quad
B_t^{H_2}  = \int_0^t K_{H_2}(t,s) dW_s.
\end{align*}
Moreover, $b\colon [0,T] \times \rr \to \rr$ is a  measurable function.
 
\begin{definition}
  A weak solution to the stochastic differential equation \eqref{one2} is defined as a triple $({B}^{H_1},{B}^{H_2},X)$ on a filtered probability space $(\Omega, \mathbb{F}, (\mathbb{F}_t)_{t \in [0, T]}, P)$ such that the following conditions hold:
\begin{itemize}
	\item [(i)] There exists an $\dF$-Brownian motion $W=(W_t)_{t \in [0, T]}$   such that
	\[
	B_t^{H_1}  = \int_0^t K_{H_1}(t,s) dW_s \quad \text{ and } \quad
	B_t^{H_2}  = \int_0^t K_{H_2}(t,s) dW_s , \quad t\in [0,T].
	\]
	\item [(ii)] ${B}^{H_1},$ ${B}^{H_2}$ and $X$ satisfy \eqref{one2}.
\end{itemize}
\end{definition}

The purpose  of this section is to prove existence of a weak solution
 under some assumptions on the function $b$ depending on the value of the Hurst indices. In particular, we will show that the regularization effect of the fractional Brownian motion with the smaller Hurst index dominates.     
 We will impose one of the following two conditions:
  
\begin{itemize}
	\item [(A1)] The function $b$ satisfies a linear growth condition, i.e. there exists $c \in (0, \infty)$ such that for all $(t,x) \in [0,T] \times \rr$ we have
	\[
	 |b(t,x)| \leq c (1+|x|).
	 \]
	\item [(A2)] Let $H = \min (H_1, H_2)>\frac 12$. The function $b$ satisfies the following H\"older condition: there exists $c \in (0, \infty)$ such that for all $(t,x), (s,y) \in [0,T] \times \rr$  we have
\[
 |b(t,x) - b(s,y)| \leq c (|x-y|^\beta + |t-s|^\gamma),
 \]
	where $\beta \in (1-\frac1 {2H}, 1)$ and $\gamma>H-\frac12$. 
\end{itemize}
 
 The main result is the following one.
 
 \begin{theorem}\label{weak2}
	There exists a weak solution to equation \eqref{one2} if one of the following conditions holds:
	\begin{itemize}
		\item [\textnormal{(i)}] $\min (H_1, H_2) \le \frac12$ and \textnormal{(A1)} is true.
		\item [\textnormal{(ii)}] $\min (H_1, H_2) > \frac12$ and \textnormal{(A2)} is true.
	\end{itemize}
\end{theorem}

\begin{proof}
Let $W=(W_t)_{t\in 0,T]}$ be a standard Brownian motion defined on a probability space $(\Omega, \mathbb{F}, P)$. We denote by  $(\dF_{t} )_{t\in [0,T]}$ the natural filtration generated by $W$. Define  the fractional Brownian motions 
  $B^{H_1}=(B^{H_1}_t)_{t \in [0, T]}$ and $B^{H_2}=(B^{H_2}_t)_{t \in [0, T]}$ by
\begin{align*}
B_t^{H_1}  = \int_0^t K_{H_1}(t,s) dW_s \quad \text{ and } \quad
B_t^{H_2}  = \int_0^t K_{H_2}(t,s) dW_s.
\end{align*}
 Suppose that we can find two  $\dF$-adapted  processes  $(u_t)_{t \in [0, T]}$ and  $(v_t)_{t \in [0, T]}$ with integrable paths  such that     $\int_0^\bullet u_s ds \in  I_{0+}^{H_1+\frac{1}{2}} ( L^2 ([0,T]) )$ and $\int_0^\bullet v_s ds \in  I_{0+}^{H_2+\frac{1}{2}} ( L^2 ([0,T]) )$ 	 almost surely and the following conditions are satisfied:
 \begin{itemize}
 \item[(i)] For all $t\in [0,T]$, we have
 \[
 u_t+v_t = b\left( t, x_0 + B_t^{H_1} + B_t^{H_2}\right)
 \]
 \item[(ii)]  For all $t\in [0,T]$ we have
  \[
  \psi_t:= \left( K_{H_1}^{-1} \int_0^\bullet u_r dr \right) (t)=\left( K_{H_2}^{-1} \int_0^\bullet v_r dr \right) (t)
	\]
 and the process $(\psi_t)_{t \in [0, T]}$  
	satisfies the Novikov condition
	\begin{equation} \label{Nov}
	 \mathbb{E} \left[ \exp \left( \int_0^t \frac12 \int_0^t \psi_s^2 ds\right) \right] < \infty, \quad t \in (0,T].
	 \end{equation}
	 \end{itemize}
Then,  $u,v$ satisfy the conditions of Theorem \ref{GT2} and we can consider  the probability measure $Q$ as given in Theorem \ref{GT2}. In this context, setting $(\wt{B}^{H_1}_t)_{t \in [0, T]}$,  $(\wt{B}^{H_2}_t)_{t \in [0, T]}$ and $X=(X_t)_{t\in [0,T]}$, with
\[
 \wt{B}^{H_1}_t = {B}^{H_1}_t  -\int_0^t u_s ds \quad \text{and} \quad \wt{B}_t^{H_2} = B_t^{H_2} -\int_0^t v_s ds, \quad t \in [0,T],
 \]
 and $X_t= x_0+ B^{H_1} _ t+B^{H_2} _ t$, 
the triple $(\wt{B}^{H_1},\wt{B}^{H_2},  X)$ is a weak solution to equation \eqref{one2} on the filtered probability space
$(\Omega, \mathbb{F}, (\mathbb{F}_t)_{t \in [0, T]}, Q)$. In fact, for every $t \in [0, T]$ we have
\begin{align*}
X_t &= x_0+ B^{H_1}_t + B^{H_2}_t = \wt{B}^{H_1}_t + \wt{B}^{H_2}_t + \int_0^t \left( u_s + v_s \right) ds \\
& =x_0+ \wt{B}^{H_1}_t +   \wt{B}^{H_2}_t + \int_0^t  b(s, B_s^{H_1} + B_s^{H_2}) ds.
\end{align*}
Furthermore, by our assumption (i) and by construction,   there exists a standard  $\dF$-Brownian motion $\wt{W}=(\wt{W}_t)_{t \in [0, T]}$ with respect to $Q$ such that
\[
\wt{B}_t^{H_1}  = \int_0^t K_{H_1}(t,s) d\wt{W}_s \quad \text{ and } \quad
\wt{B}_t^{H_2}  = \int_0^t K_{H_2}(t,s) d\wt{W}_s , \quad t\in [0,T].
\]
Notice that if $H_1=\frac 12$, then in the above construction $B^{H_1} =W$ and $\psi=u$.

To complete the proof it remains to show the existence of  processes $u$ and $v$ satisfying the above conditions (i) and (ii).
The proof of this fact will be decomposed into several cases.

 \subsection{Case $H_1=\frac 12$}
 Suppose $H_1=\frac 12$, that is, $B^{H_1} =W$ and  put $H_2=H$ to simplify the presentation.
   First assume that $H<\frac12$. Define $\alpha = \frac12 - H$. Let $(v_t)_{t \in [0, T]}$ be the process defined by
	\begin{align*}
	v_t &= b(t, x_0+W_t + B_t^H) + t^{-\alpha}\sum_{k=1}^{\infty}  \frac{(-1)^k}{\Gamma(k\alpha)} \int_0^t (t-s)^{k\alpha -1} s^{\alpha} b(s, x_0+W_s + B_s^H) ds  \\
	& =  t^{-\alpha}\sum_{n=0}^\infty (-1)^n I_{0+}^{n\alpha}  \big( t^\alpha b(t, x_0+W_t + B_t^H)\big) \\
	& = t^{-\alpha}[1+ I_{0+}^{\alpha}]^{-1} \big( t^\alpha b(t, x_0+W_t + B_t^H)\big)
	\end{align*}
	and
	let $(\psi_t)_{t \in [0, T]}$ be given by
	\[
	 \psi_t=t^{-\alpha} I_{0+}^{\alpha} t^{\alpha} v_t = t^{-\alpha} \frac{1}{\Gamma (\alpha)} \int_0^t (t-s)^{\alpha-1} s^{\alpha} v_s ds 
	 \]
	for every $t \in (0,T]$.  Clearly, $\psi_t + v_t=  b(t, x_0+ W_t+ B_t^H)$.
	Using the assumption that $b$ is of linear growth it is not difficult to prove that Novikov's condition \eqref{Nov} is fulfilled, so that 
	the processes $u:=\psi$ and $v$ satisfy the above properties (i) and (ii).
	
	 Now suppose that $H>\frac12$. Define $\alpha = H-\frac12$. Here, we define the processes $(\psi_t)_{t \in [0, T]}$ and $(v_t)_{t \in [0, T]}$ by
	\begin{align*}
	\psi_t = t^{\alpha} \sum_{k=1}^{\infty}  \frac{(-1)^k}{\Gamma(k\alpha)} \int_0^t (t-s)^{k\alpha -1} s^{-\alpha} b(s, x_0+W_s + B_s^H) ds +  b(t, x_0+W_t + B_t^H) 
	\end{align*}
	and
	\[
	v_t = t^{\alpha} \frac{1}{\Gamma (\alpha )} \int_0^t (t-s)^{\alpha-1} s^{-\alpha} \psi_s ds 
	\]
	for every $t \in (0,T]$. Then arguing as in the case $H<\frac12$ above we obtain a weak solution to \eqref{one2}. This finishes the proof in the present case.

 \subsection{Case $\min (H_1, H_2) < \frac12$ and $\max (H_1, H_2) > \frac12$}
Without loss of generality assume that $H_1 < H_2$, i.e. $H_1 < \frac12$ and $H_2 > \frac12$. Define $\alpha_1 = \frac12 - H_1$, $\alpha_2 = H_2-\frac12$. From now on, we let $b_t :=  b(t, x_0+B_t^{H_1} + B_t^{H_2})$, $t \in [0,T]$. In this case we define the processes $u,v$ by
\begin{align}  \notag
u_t & = b_t + t^{-\alpha_1}\sum_{n=1}^{\infty} (-1)^n \left[ t^{\alpha_1+\alpha_2} (I_{0+}^{\alpha_2} [t^{-\alpha_1-\alpha_2} I_{0+}^{\alpha_1}])\right]^n  \left( t^{\alpha_1} b_t \right) \\  \label{ecu3}
& = t^{-\alpha_1}\sum_{n=0}^{\infty} (-1)^n \left[ t^{\alpha_1+\alpha_2} (I_{0+}^{\alpha_2} [t^{-\alpha_1-\alpha_2} I_{0+}^{\alpha_1}])\right]^n  \left( t^{\alpha_1} b_t \right) \\  \notag
&= t^{-\alpha_1}\left[ 1+t^{\alpha_1+\alpha_2} I_{0+}^{\alpha_2} [t^{-\alpha_1-\alpha_2} I_{0+}^{\alpha_1}]\right] ^{-1}   \left( t^{\alpha_1} b_t   \right)
\end{align}
and
\[
v_t = t^{\alpha_2} I_{0+}^{\alpha_2} [t^{-\alpha_1-\alpha_2} I_{0+}^{\alpha_1} (t^{\alpha_1} u_t)]
\]
for every $t \in (0,T]$.  It is easy to see that $u_t + v_t=  b(t, x_0+ W_t+ B_t^H)$.  By \eqref{pG} and  the definition of the process $v$ we have
\[
 \Big( K_{H_1}^{-1} \int_0^\bullet u_r dr \Big) (t)=\Big( K_{H_2}^{-1} \int_0^\bullet v_r dr \Big) (t),
 \]
for every $t \in (0,T]$. Set
\[
\psi_t := \left( K_{H_1}^{-1} \int_0^\bullet u_r dr \right) (t) , \quad t \in (0,T].
\]
It now suffices to show that $(\psi_t)_{t \in (0,T]}$ satisfies Novikov's condition \eqref{Nov}. Using Lemma \ref{bound1} below  we get that
\[
 |u_t| \leq c_T \left( 1+ \| B^{H_1} + B^{H_2} \|_\infty \right) , \quad t \in (0,T],
 \]
  where $\|\cdot\|$ denotes the supremum norm on $[0,T]$.
Then the proof of Novikov's condition follows the lines of the proof of \cite[Theorem 3]{NO} on page 109, completing the proof of Theorem \ref{weak2} in the present case. \hfill $\Box$
\begin{lemma} \label{bound1}
	  It holds
	\[
	 |u_t| \leq c_T \left( 1+ \| B^{H_1} + B^{H_2} \|_\infty \right), \quad t \in (0,T],
	 \]
	for some constant $c_T \in (0, \infty)$ which depends on $T$, $H_1$, $H_2$, and the constant $c$ appearing in the linear growth condition (A1).
\end{lemma}

\begin{proof}
	Let $t\in (0,T]$. Note first that, using the linear growth condition  (A1), we obtain
	\begin{align*}
	& \left| \left[ t^{\alpha_1+\alpha_2} (I_{0+}^{\alpha_2} [t^{-\alpha_1-\alpha_2} I_{0+}^{\alpha_1}])\right]^n  \left( t^{\alpha_1} b(t, B_t^{H_1} + B_t^{H_2}) \right) \right| \leq  c\left( 1+ \| B^{H_1} + B^{H_2} \|_\infty \right) \\
	& \quad \times \left[ t^{\alpha_1+\alpha_2} (I_{0+}^{\alpha_2} [t^{-\alpha_1-\alpha_2} I_{0+}^{\alpha_1}])\right]^n  ( t^{\alpha_1}  ).
	\end{align*}
We will show, by induction, that for all $n\in \N$, $n \geq 2$,
	\begin{equation} 
	  \left[ t^{\alpha_1+\alpha_2} (I_{0+}^{\alpha_2} [t^{-\alpha_1-\alpha_2} I_{0+}^{\alpha_1}])\right]^n  ( t^{\alpha_1})   =
	   \frac {\Gamma(\alpha_1-\alpha_2+1)}  {\Gamma((n+1) \alpha_1+ (n-1)\alpha_2+1)}
	t^{ (n+1)\alpha_1+ n\alpha_2}.
	 \label{pb1}
	\end{equation} 
	Using the equality
	\[
	  I_{0+}^{\alpha} [t^{\beta}] =   \frac { \Gamma(\beta+1)}{\Gamma (\alpha+\beta+1)}  t^{\alpha+\beta} 
\]
	we obtain, for every $\beta >\alpha_2$, 
	\begin{align}   \label{ecu2}
	\left[ t^{\alpha_1+\alpha_2} (I_{0+}^{\alpha_2} [t^{-\alpha_1-\alpha_2} I_{0+}^{\alpha_1}])\right]  ( t^{\beta}  ) & = 
	  \frac { \Gamma(\beta+1)}{\Gamma (\alpha_1+\beta+1)} 
	    t^{\alpha_1+\alpha_2} (I_{0+}^{\alpha_2} [t^{\beta-\alpha_2}]) =  \frac  {\Gamma(\beta-\alpha_2+1)}{\Gamma (\alpha_1+\beta+1)}   
	    t^{\alpha_1+ \alpha_2+\beta}.
	\end{align}
	In particular, taking $\beta=\alpha_1$ in \eqref{ecu2}, we get
\[
	\left[ t^{\alpha_1+\alpha_2} (I_{0+}^{\alpha_2} [t^{-\alpha_1-\alpha_2} I_{0+}^{\alpha_1}])\right]  ( t^{\alpha_1}  )  =  \frac  {\Gamma(\alpha_1-\alpha_2+1)}{\Gamma (2\alpha_1+1)}   
	    t^{2\alpha_1+\alpha_2}.
\]
Thus,  we get
	\begin{align*}
	\left[ t^{\alpha_1+\alpha_2} (I_{0+}^{\alpha_2} [t^{-\alpha_1-\alpha_2} I_{0+}^{\alpha_1}])\right]^2  ( t^{\alpha_1}  )  =
	 \frac  {\Gamma(\alpha_1-\alpha_2+1)}{\Gamma (2\alpha_1+1)}    \left[ t^{\alpha_1+\alpha_2} (I_{0+}^{\alpha_2} [t^{-\alpha_1-\alpha_2} I_{0+}^{\alpha_1}])\right]  ( t^{2\alpha_1 + \alpha_2}  ) 
	\end{align*}
		and putting $\beta=2\alpha_1+\alpha_2$ in \eqref{ecu2}, yields
	\begin{align*}
	\left[ t^{\alpha_1+\alpha_2} (I_{0+}^{\alpha_2} [t^{-\alpha_1-\alpha_2} I_{0+}^{\alpha_1}])\right]^2  ( t^{\alpha_1}  )  =
	 \frac  {\Gamma(\alpha_1-\alpha_2+1)}{\Gamma (3\alpha_1+\alpha_2+1)}    ( t^{3\alpha_1 + 2\alpha_2}  ).
	\end{align*}
	which proves \eqref{pb1} for $n=2$.
	
	 Now assume that \eqref{pb1} holds for some $n \in \N$ with $n \geq 2$.  We can write
\begin{align*}
	 \left[ t^{\alpha_1+\alpha_2} (I_{0+}^{\alpha_2} [t^{-\alpha_1-\alpha_2} I_{0+}^{\alpha_1}])\right]^{n+1}  ( t^{\alpha_1})  
	 &= \frac {\Gamma(\alpha_1-\alpha_2+1)}  {\Gamma((n+1) \alpha_1+ (n-1)\alpha_2+1)} \\
	& \times  \left[ t^{\alpha_1+\alpha_2} (I_{0+}^{\alpha_2} [t^{-\alpha_1-\alpha_2} I_{0+}^{\alpha_1}])\right] (t^{ (n+1)\alpha_1+ n\alpha_2}).
	 \end{align*}
 Applying \eqref{ecu2} with $ \beta=(n+1)\alpha_1+ n\alpha_2$, we obtain \eqref{pb1} for $n+1$.
 
 Finally, using the representation of the process $u$ given in   \eqref{ecu3}, we obtain
 \begin{align*}
 |u_t| &  \le  c  \left( 1+ \| B^{H_1} + B^{H_2} \|_\infty \right) t^{-\alpha_1}  \sum_{n=0}^\infty  \frac {\Gamma(\alpha_1-\alpha_2+1)}  {\Gamma((n+1) \alpha_1+ (n-1)\alpha_2+1)}
	t^{ (n+1)\alpha_1+ n\alpha_2}\\
	&\le c_T  \left( 1+ \| B^{H_1} + B^{H_2} \|_\infty \right).
	\end{align*}
This completes the  proof of this Lemma.
\end{proof}

\subsection{Case $H_1 < \frac12$ and $H_2 <\frac12$}
Define $\alpha_1 = \frac12 - H_1$ and $\alpha_2 = \frac12-H_2$, without loss of generality we assume that $\alpha_2>\alpha_1$.   In this case we define the processes $u,v$ by
\begin{align}  \notag
v_t &=  b_t  + t^{-\alpha_1}\sum_{n=1}^{\infty} (-1)^n  \left(  D_{0+}^{\alpha_1}t^{\alpha_1-\alpha_2} I_{0+}^{\alpha_2} t^{\alpha_2-\alpha_1}\right)^n(t^{\alpha_1} b_t) \\  \label{ecu4}
& = t^{-\alpha_1} \sum_{n=0}^{\infty} (-1)^n \left(  D_{0+}^{\alpha_1}t^{\alpha_1-\alpha_2} I_{0+}^{\alpha_2} t^{\alpha_2-\alpha_1}\right)^n(t^{\alpha_1} b_t) \\  \notag
& = t^{-\alpha_1} \left[ 1+ D_{0+}^{\alpha_1}t^{\alpha_1-\alpha_2} I_{0+}^{\alpha_2} t^{\alpha_2-\alpha_1}\right] ^{-1}   \left( t^{\alpha_1} b_t  \right)
\end{align}
and
\[
 u_t = t^{-\alpha_1}D_{0+}^{\alpha_1}t^{\alpha_1-\alpha_2} I_{0+}^{\alpha_2} t^{\alpha_2} v_t
 \]
for every $t \in (0,T]$.  

Clearly, $u_t + v_t=  b(t, x_0+ W_t+ B_t^H)$.
By \eqref{pG} and definition of the process $u$ we have
\[
 \left( K_{H_1}^{-1} \int_0^\bullet u_r dr \right) (t)=\left( K_{H_2}^{-1} \int_0^\bullet v_r dr \right) (t),
 \]
for every $t \in (0,T]$. 
Now let
\begin{equation}  \label{ecu3}
\psi_t := \left( K_{H_2}^{-1} \int_0^\bullet v_r dr \right) (t) , \quad t \in (0,T].
\end{equation}
It suffices to show that $(\psi_t)_{t \in (0,T]}$ satisfies Novikov's condition. Let $t \in (0,T]$. 
From   \eqref{ecu3}, \eqref{pG} and  \eqref{ecu4},  and the  linearity of fractional integration, we can write
\begin{align*}
\psi_t &= t^{-\alpha_2} I_{0+}^{\alpha_2} t^{\alpha_2} (v_t)  \\
&= t^{-\alpha_2}
I_{0+}^{\alpha_2} t^{\alpha_2} ( b_t) + t^{-\alpha_2}  I_{0+}^{\alpha_2} t^{\alpha_2-\alpha_1}  \left(\sum_{n=1}^{\infty} (-1)^n  \left(  D_{0+}^{\alpha_1}t^{\alpha_1-\alpha_2} I_{0+}^{\alpha_2} t^{\alpha_2-\alpha_1}\right)^n(t^{\alpha_1} b_t) \right).
\end{align*}
Therefore,
\[
|\psi_t | \le  \left|  t^{-\alpha_2} I_{0+}^{\alpha_2} t^{\alpha_2} ( b_t) \right| + J_t,
\]
where
\[
J_t := \left|  t^{-\alpha_2} I_{0+}^{\alpha_2} t^{\alpha_2-\alpha_1}  \left(\sum_{n=1}^{\infty} (-1)^n  \left(  D_{0+}^{\alpha_1}t^{\alpha_1-\alpha_2} I_{0+}^{\alpha_2} t^{\alpha_2-\alpha_1}\right)^n(t^{\alpha_1} b_t) \right) \right|.
\]
In the following we only estimate $J_t$, since the other term $|  t^{-\alpha_2}I_{0+}^{\alpha_2} t^{\alpha_2} ( b_t) |$ is estimated as above.
Recall that $\alpha_2 - \alpha_1>0$. With the linearity of fractional integration and the triangle inequality we get
\[
J_t \leq \left(\sum_{n=1}^{\infty} \left|  t^{-\alpha_2} I_{0+}^{\alpha_2}t^{\alpha_2-\alpha_1} \left( D_{0+}^{\alpha_1}t^{\alpha_1-\alpha_2} I_{0+}^{\alpha_2} t^{\alpha_2-\alpha_1}\right)^n(t^{\alpha_1} b_t) \right| \right).
\]
For each   $n\in\N$ we are going to  estimate the term
\[
J_t^n := \left|  t^{-\alpha_2} I_{0+}^{\alpha_2}t^{\alpha_2-\alpha_1} \Big( D_{0+}^{\alpha_1}t^{\alpha_1-\alpha_2} I_{0+}^{\alpha_2} t^{\alpha_2-\alpha_1}\Big)^n(t^{\alpha_1} b_t) \right|.
\]
First observe that for a function $(g_t)_{t \in (0,T]}$, in view of formula  \eqref{we}, we have
\[ D_{0+}^{\alpha_1} t^{\alpha_1-\alpha_2} g_t= t^{\alpha_1-\alpha_2}D_{0+}^{\alpha_1}  g_t + \frac {\alpha_1}{ \Gamma(1-\alpha_1)} \int_0^t g_s \frac{t^{\alpha_1-\alpha_2}-s^{\alpha_1-\alpha_2}}{(t-s)^{\alpha_1+1} }ds.
\]
Thus,
\begin{align} \label{c2}
\begin{split}
J_t^n & \leq \left|   t^{-\alpha_2} I_{0+}^{\alpha_2-\alpha_1} I_{0+}^{\alpha_2}t^{\alpha_2-\alpha_1} \left( D_{0+}^{\alpha_1}t^{\alpha_1-\alpha_2} I_{0+}^{\alpha_2} t^{\alpha_2-\alpha_1}\right)^{n-1}(t^{\alpha_1} b_t)\right| \\
& +   
\frac {\alpha_1}{ \Gamma(1-\alpha_1)}
\left|  t^{-\alpha_2}  I_{0+}^{\alpha_2}t^{\alpha_2-\alpha_1} \int_0^t I_{0+}^{\alpha_2}s^{\alpha_2-\alpha_1} \left( D_{0+}^{\alpha_1}s^{\alpha_1-\alpha_2} I_{0+}^{\alpha_2} s^{\alpha_2-\alpha_1}\right)^{n-1}(s^{\alpha_1} b_s) \frac{t^{\alpha_1-\alpha_2} - s^{\alpha_1-\alpha_2}}{(t-s)^{\alpha_1+1}} ds\right| \\
& \leq   t^{-\alpha_2}  I_{0+}^{\alpha_2-\alpha_1}  t^{\alpha_2}J_t^{n-1} + \frac {\alpha_1}{ \Gamma(1-\alpha_1)}  t^{-\alpha_2}  I_{0+}^{\alpha_2}t^{\alpha_2-\alpha_1} \int_0^t s^{\alpha_2}  J_s^{n-1} \frac{ |t^{\alpha_1-\alpha_2} - s^{\alpha_1-\alpha_2}|  }{(t-s)^{\alpha_1+1}} ds \\
&\le    I_{0+}^{\alpha_2-\alpha_1}  J_t^{n-1} +  \frac {\alpha_1}{ \Gamma(1-\alpha_1)}   I_{0+}^{\alpha_2}t^{-\alpha_1} \int_0^t s^{\alpha_2}  J_s^{n-1} \frac{|t^{\alpha_1-\alpha_2} - s^{\alpha_1-\alpha_2}|}{(t-s)^{\alpha_1+1}} ds,
\end{split}
\end{align}
with the convention that $J_t^0=  | t^{-\alpha_2}  I_{0+}^{\alpha_2} t^{\alpha_2} b_t|$.   Notice that
\[
J^0_t \le |   I_{0+}^{\alpha_2}   b_t| \le \left( \sup_{t\in [0,T]} |b_t|  \right) (I_{0+}^{\alpha_2} 1)_t= \frac {
t^{\alpha_2} }{\alpha_2\Gamma(\alpha_2)}  \sup_{t\in [0,T]} |b_t|    .
\]
We claim that  there exist constants $C$ and $\kappa$  such that
\begin{equation} \label{c1}
J_t^n \leq  C\kappa^{n-1} \left( \sup_{t\in [0,T]} |b_t|  \right)  \frac{t^{(n+1)\alpha_2 - n\alpha_1}}{\Gamma \left( (n+1)\alpha_2-n\alpha_1) \right)}
\end{equation}
for every $n\ge 1$. We prove this claim by induction. For $n=1$ we obtain from \eqref{c2} that
\begin{align*}
J_t^1 & \leq  \frac{1}{\alpha_2 \Gamma(\alpha_2)}\left( \sup_{t\in [0,T]}  |b_t|  \right) \left[     (I_{0+}^{2\alpha_2-\alpha_1}1)_t + \frac {\alpha_1}{ \Gamma(1-\alpha_1)}   I_{0+}^{\alpha_2}t^{-\alpha_1} \int_0^t s^{2\alpha_2} \frac{|t^{\alpha_1-\alpha_2} - s^{\alpha_1-\alpha_2}|}{(t-s)^{\alpha_1+1}} ds 
\right] \\
& \leq   \frac{1}{\alpha_2 \Gamma(\alpha_2)}\left( \sup_{t\in [0,T]} |b_t|  \right)
\left[ (I_{0+}^{2\alpha_2-\alpha_1}1)_t   +  c_1  I_{0+}^{\alpha_2} t^{\alpha_2 -\alpha_1}\right] \\
& \leq  C \left( \sup_{t\in [0,T]} |b_t|  \right)  t^{2\alpha_2-\alpha_1},
\end{align*}
where 
\[
c_1=  \frac {\alpha_1}{ \Gamma(1-\alpha_1)}\int_0^1 x^{2\alpha_2}  \frac { 1-x^{\alpha_1-\alpha_2} }
{(1-x) ^{\alpha_1+1}}dx
\]
and
\[
C= \frac{1}{\alpha_2 \Gamma(\alpha_2)} \Bigg( \frac 1{ \Gamma(2\alpha_2-\alpha_1+1)} +c_1  \frac { \Gamma(\alpha_1) \Gamma(\alpha_2-\alpha_1+1)}{\Gamma(\alpha_2+1)\Gamma(\alpha_2)} \Bigg).
\]
Therefore,   \eqref{c1} is true for $n=1$. Now assume that \eqref{c1} holds for some $n\in\N$. Then using \eqref{c2} again we get
\begin{align*}
J_t^{n+1} & \leq 
 \frac {C\kappa^{n-1} \left( \sup_{t\in [0,T]} |b_t|  \right)}
 {\Gamma \left( (n+1)\alpha_2-n\alpha_1) \right)}
   \left[   
 I_{0+}^{\alpha_2-\alpha_1} t^{(n+1)\alpha_2-n\alpha_1 } +  I_{0+}^{\alpha_2}t^{-\alpha_1} \int_0^t s^{(n+2)\alpha_2-n\alpha_1} \frac{|t^{\alpha_1-\alpha_2} - s^{\alpha_1-\alpha_2}|}{(t-s)^{\alpha_1+1}} ds  \right]\\
& \leq  \frac { C\kappa^{n-1} \left( \sup_{t\in [0,T]} |b_t|  \right)}
 {\Gamma \left( (n+1)\alpha_2-n\alpha_1) \right)}
   \left[  I_{0+}^{\alpha_2-\alpha_1} t^{(n+1)\alpha_2-n\alpha_1 }+ c_2 I_{0+}^{\alpha_2}t^{(n+1) (\alpha_2-\alpha_1)} \right],
   \end{align*}
   where
   \[
   c_2=  \frac{\alpha_1}{ \Gamma(1-\alpha_1)}\int_0^1 \frac { 1-x^{\alpha_1-\alpha_2}} { (1-x) ^{\alpha_1+1}}d x.
   \]
   Notice that
   \begin{align*}
   I_{0+}^{\alpha_2-\alpha_1} t^{(n+1)\alpha_2-n\alpha_1 }  &=
      \frac { \Gamma((n+1) \alpha_2 -n\alpha_1 +1)} { \Gamma((n+2) \alpha_2 - (n+1) \alpha_1+1)} t^{(n+2) \alpha_2- (n+1) \alpha_1} \\
     & \le  \frac { \Gamma((n+1) \alpha_2 -n\alpha_1 )} { \Gamma((n+2) \alpha_2 - (n+1) \alpha_1)} t^{(n+2) \alpha_2- (n+1) \alpha_1}
\end{align*}
and
    \begin{align*}
I_{0+}^{\alpha_2}t^{(n+1) (\alpha_2-\alpha_1)} & =
      \frac { \Gamma((n+1) (\alpha_2 -\alpha_1 )+1)} { \Gamma((n+2) \alpha_2 - (n+1) \alpha_1+1)} t^{(n+2) \alpha_2- (n+1) \alpha_1} \\
      &\le   \frac { \Gamma((n+1) \alpha_2 -n\alpha_1)} { \Gamma((n+2) \alpha_2 - (n+1) \alpha_1)} t^{(n+2) \alpha_2- (n+1) \alpha_1} 
\end{align*}
This implies
\[
J_t^{n+1}  \leq 
\frac {C \kappa^{n-1} \left( \sup_{t\in [0,T]} |b_t|  \right)}
 {\Gamma \left( (n+2)\alpha_2-(n+1)\alpha_1) \right)} (1+c_2)  t^{(n+2) \alpha_2- (n+1) \alpha_1} 
\]
proving \eqref{c1} for every $n\in\N$ with the constant $\kappa=  1+c_2$. Hence, altogether we obtain that the inequality
\[
 J_2 \leq c  t^{\alpha_2}\sup_{t\in [0,T]} |b_t|
 \]
is true for some constant $c \in (0, \infty)$. The rest of the proof of Novikov's condition follows exactly by the same argument as in the first case above. \hfill $\Box$

\subsection{Case $H_1 > \frac12$ and $H_2 >\frac12$}
Define $\alpha_1 = H_1-\frac12 $ and $\alpha_2 =H_2- \frac12$, and we assume without loss of generality that 
\begin{equation*}
\alpha_1-\alpha_2 >0, \quad \text{ i.e. } \quad H_1>H_2
\end{equation*}
holds. Recall that now we assume that (A2) holds. Before constructing a weak solution we first state the following, which will be made used of in proving that the assumptions of Theorem \ref{GT2} hold. As before, throughout this section we denote $b_t :=  b(t, B_t^{H_1} + B_t^{H_2})$ for every $t \in [0,T]$.
\begin{lemma}\label{bound2}
	Under the present conditions for $\alpha \in (0, \infty)$ it holds
	\[
	D_{0+}^{\alpha_2}(t^{-\alpha}b_t)\leq   G \left[ t^{-\alpha-\alpha_2} + t^{\gamma - \alpha_2 - \alpha} + t^{-\alpha+\beta(H_2 - \varepsilon)-\alpha_2} \right] , \quad t \in (0,T], 
	\]
	for every $\varepsilon \in (0, \infty)$  such that $\beta (H_2- \varepsilon) > H_2-\frac 12$ and some positive random variable $G$ with
	\begin{equation} \label{ecu8}
	\mathbb{E}[\exp (\lambda  G^{2})] <\infty\qquad   \text{for all} \quad \lambda >0.
		\end{equation}
	 In particular, for all $ t \in (0,T]$, we have
	\[
	D_{0+}^{\alpha_2}(t^{-\alpha}b_t)\leq  G (1+ T^\gamma +T^{\beta(H_2-\varepsilon)}) t^{-\alpha-\alpha_2}=: G_1t^{-\alpha-\alpha_2},
	\]
	where $G_1$ satisfies \eqref{ecu8}.
\end{lemma}
\begin{proof}
	The proof follows the calculations on \cite[p.109-110]{NO}.
\end{proof}

Now we define the processes $u,v$ by
\begin{align}  \notag
v_t &=  b_t + t^{\alpha_1}\sum_{n=1}^{\infty} (-1)^n  \left(  I_{0+}^{\alpha_1}t^{\alpha_2-\alpha_1} D_{0+}^{\alpha_2} t^{\alpha_1-\alpha_2}\right)^n(t^{-\alpha_1} b_t) \\  \label{ecu7}
& = t^{\alpha_1} \sum_{n=0}^{\infty} (-1)^n \left(  I_{0+}^{\alpha_1}t^{\alpha_2-\alpha_1} D_{0+}^{\alpha_2} t^{\alpha_1-\alpha_2}\right)^n(t^{-\alpha_1} b_t) \\  \notag
& = t^{\alpha_1} \left[ 1+ I_{0+}^{\alpha_1}t^{\alpha_2-\alpha_1} D_{0+}^{\alpha_2} t^{\alpha_1-\alpha_2}\right] ^{-1}   \left( t^{-\alpha_1} b_t  \right)
\end{align}
and
\[
 u_t = t^{\alpha_1}I_{0+}^{\alpha_1}t^{\alpha_2-\alpha_1} D_{0+}^{\alpha_2} t^{-\alpha_2} v_t
 \]
for every $t \in (0,T]$.   

Clearly, $u_t + v_t=  b(t, x_0+ W_t+ B_t^H)$. By \eqref{pG} and definition of the process $u$ we have
\[
 \psi_t:=\left( K_{H_1}^{-1} \int_0^\bullet u_r dr \right) (t)=\left( K_{H_2}^{-1} \int_0^\bullet v_r dr \right) (t),
 \]
for every $t \in (0,T]$. From  \eqref{ecu7} and \eqref{pG}, we can write
\[
 \psi_t =   t^{\alpha_2} D_{0+}^{\alpha_2} t^{-\alpha_2} \left( b_t +  t^{\alpha_1}\sum_{n=1}^{\infty} (-1)^n  \left(   I_{0+}^{\alpha_1}t^{\alpha_2-\alpha_1} D_{0+}^{\alpha_2} t^{\alpha_1-\alpha_2}\right)^n(t^{-\alpha_1} b_t) \right) 
 \]
for every $t \in (0,T]$. It remains to show that $(\psi_t)_{t \in (0,T]}$ satisfies Novikov's condition. Let $t \in (0,T]$ and write
$\psi_t =  t^{\alpha_2} D_{0+}^{\alpha_2} t^{-\alpha_2} b_t + K_t$
with
\begin{align*}
K_t&:= t^{\alpha_2} D_{0+}^{\alpha_2} t^{\alpha_1-\alpha_2} \left(\sum_{n=1}^{\infty} (-1)^n  \left(   I_{0+}^{\alpha_1}t^{\alpha_2-\alpha_1} D_{0+}^{\alpha_2} t^{\alpha_1-\alpha_2}\right)^n(t^{-\alpha_1} b_t) \right) \\
&=  \sum_{n=1}^{\infty} (-1)^n  t^{\alpha_2} D_{0+}^{\alpha_2} t^{\alpha_1-\alpha_2}\Big(   I_{0+}^{\alpha_1}t^{\alpha_2-\alpha_1} D_{0+}^{\alpha_2} t^{\alpha_1-\alpha_2}\Big)^n(t^{-\alpha_1} b_t).
\end{align*}
Using Lemma \ref{bound2}, we obtain
\begin{equation}  \label{a2}
|t^{\alpha_2} D_{0+}^{\alpha_2} t^{-\alpha_2} b_t | \le G t^{-\alpha_2},
\end{equation}
where  $G$ is a random variable that satisfies  \eqref{ecu8}.

We now estimate $|K_t^n|$, $n \in \N$, for
\[
K_t^n:= t^{\alpha_2} D_{0+}^{\alpha_2} t^{\alpha_1-\alpha_2}\Big(   I_{0+}^{\alpha_1}t^{\alpha_2-\alpha_1} D_{0+}^{\alpha_2} t^{\alpha_1-\alpha_2}\Big)^n(t^{-\alpha_1} b_t).
\]
First observe that for a function $(g_t)_{t \in (0,T]}$ we have
\[
 D_{0+}^{\alpha_2} t^{\alpha_1-\alpha_2} g_t= t^{\alpha_1-\alpha_2}D_{0+}^{\alpha_2}  g_t + \frac {\alpha_2}{ \Gamma(1-\alpha_2)} \int_0^t g_s \frac{t^{\alpha_1-\alpha_2}-s^{\alpha_1-\alpha_2}}{(t-s)^{\alpha_2+1} }ds.
 \]
We thus obtain the inequality
\[
\left| D_{0+}^{\alpha_2} t^{\alpha_1-\alpha_2} g_t \right| \leq  \left|t^{\alpha_1-\alpha_2}D_{0+}^{\alpha_2}  g_t\right| + 
\frac {\alpha_2}{ \Gamma(1-\alpha_2)} \int_0^t |g_s| \frac{t^{\alpha_1-\alpha_2}-s^{\alpha_1-\alpha_2}}{(t-s)^{\alpha_2+1} }ds. 
\]
Hence, for $K_t^n$ we get the estimate
\begin{align}   \notag
|K_t^n| &\leq \left| t^{\alpha_1} D_{0+}^{\alpha_2}\left(   I_{0+}^{\alpha_1}t^{\alpha_2-\alpha_1} D_{0+}^{\alpha_2} t^{\alpha_1-\alpha_2}\right)^n(t^{-\alpha_1} b_t)\right| \\  \notag
& \quad +\frac {\alpha_2}{ \Gamma(1-\alpha_2)}     t^{\alpha_2} \int_0^t  \left|  \left( I_{0+}^{\alpha_1}s^{\alpha_2-\alpha_1} D_{0+}^{\alpha_2} s^{\alpha_1-\alpha_2}\right)^n(s^{-\alpha_1} b_s)  \right| \frac{t^{\alpha_1-\alpha_2}-s^{\alpha_1-\alpha_2}}{(t-s)^{\alpha_2+1} }ds  \\
& =: |K^{n,1}_t| +\frac {\alpha_2}{ \Gamma(1-\alpha_2)}    |K^{n,2}_t|.  \label{a1}
\end{align}

\begin{lemma}\label{bound3}
	For every $n \in \N$ it holds
	\begin{equation*}
	|K_t^{n,1}| \leq \frac{G C^n}{\Gamma(n(\alpha_1-\alpha_2))} t^{n(\alpha_1-\alpha_2)-\alpha_2},
	\end{equation*}
	where $G$ is a random variable satisfying the condition in Lemma \ref{bound2} and $C$ is a positive constant.
\end{lemma}
\begin{proof}
	We give a proof by induction.  	First we estimate $|K_t^{n,1}|$ for $n=1$:
	\begin{align*}
	|K_t^{1,1}| \leq \left| t^{\alpha_1} I_{0+}^{\alpha_1-\alpha_2}  D_{0+}^{\alpha_2} (t^{-\alpha_1} b_t) \right|  +  
	\frac {\alpha_2}{ \Gamma(1-\alpha_2)}
	 t^{\alpha_1} \left| I_{0+}^{\alpha_1-\alpha_2}  t^{\alpha_2-\alpha_1} \int_0^t \frac{t^{\alpha_1-\alpha_2} - s^{\alpha_1-\alpha_2}}{(t-s)^{\alpha_2+1}} s^{-\alpha_1} b_sds \right|.
	\end{align*}
	Recall that $\alpha_1-\alpha_2>0$. Hence,  making the change of variables $x=ts$ we obtain
	\[
	\int_0^t \frac{t^{\alpha_1-\alpha_2} - s^{\alpha_1-\alpha_2}}{(t-s)^{\alpha_2+1}} s^{-\alpha_1} ds = \kappa t^{\alpha_1-2\alpha_2},
	\]
	where $\kappa = \int_0^1 \frac {1-x^{\alpha_1-\alpha_2}}{(1-x)^{\alpha_2+1}}dx$.
	So that, now using Lemma \ref{bound2}
	\begin{align*}
	|K_t^{1,1}| & \leq \left| t^{\alpha_1} I_{0+}^{\alpha_1-\alpha_2}  D_{0+}^{\alpha_2} (t^{-\alpha_1} b_t) \right|  + 
	\frac {\kappa \alpha_2   \left(\sup_{s\in [0,T]} | b_s| \right)  }{  \Gamma(1-\alpha_2)}  
	t^{\alpha_1} I_{0+}^{\alpha_1-\alpha_2} t^{-\alpha_2}  \\
	&\leq G_1  t^{\alpha_1} I_{0+}^{\alpha_1-\alpha_2} t^{-\alpha_1-\alpha_2}   +
	\frac {\kappa \alpha_2   \left(\sup_{s\in [0,T]} | b_s| \right)  }{  \Gamma(1-\alpha_2)}  
	t^{\alpha_1} I_{0+}^{\alpha_1-\alpha_2} t^{-\alpha_2}\\
	& \leq  G  t^{\alpha_1-2\alpha_2},
	\end{align*}
	where $G$ is a random variable satisfying \eqref{ecu8}.
	Now assume that the assertion of the Lemma holds for some $n\in \N$. We have
	\begin{align*}
	K_t^{n+1,1} & =  t^{\alpha_1} D_{0+}^{\alpha_2}\left(   I_{0+}^{\alpha_1}t^{\alpha_2-\alpha_1} D_{0+}^{\alpha_2} t^{\alpha_1-\alpha_2}\right)\left(   I_{0+}^{\alpha_1}t^{\alpha_2-\alpha_1} D_{0+}^{\alpha_2} t^{\alpha_1-\alpha_2}\right)^n(t^{-\alpha_1} b_t) \\
	& =t^{\alpha_1} I_{0+}^{\alpha_1-\alpha_2} t^{\alpha_2-\alpha_1} D_{0+}^{\alpha_2} t^{\alpha_1-\alpha_2} \left(   I_{0+}^{\alpha_1}t^{\alpha_2-\alpha_1} D_{0+}^{\alpha_2} t^{\alpha_1-\alpha_2}\right)^n (t^{-\alpha_1} b_t) \\
	& = t^{\alpha_1} I_{0+}^{\alpha_1-\alpha_2} D_{0+}^{\alpha_2} \left(   I_{0+}^{\alpha_1}t^{\alpha_2-\alpha_1} D_{0+}^{\alpha_2} t^{\alpha_1-\alpha_2}\right)^n (t^{-\alpha_1} b_t) \\
	& \quad +  \frac {\alpha_2}{ \Gamma(1-\alpha_2)}    t^{\alpha_1}  I_{0+}^{\alpha_1-\alpha_2} t^{\alpha_2-\alpha_1} \int_0^t \frac{t^{\alpha_1-\alpha_2}-s^{\alpha_1-\alpha_2}}{(t-s)^{\alpha_2+1}} \left(   I_{0+}^{\alpha_1}s^{\alpha_2-\alpha_1} D_{0+}^{\alpha_2} s^{\alpha_1-\alpha_2}\right)^n (s^{-\alpha_1} b_s) ds \\
	& = t^{\alpha_1} I_{0+}^{\alpha_1-\alpha_2} t^{-\alpha_1} K_t^{n,1} \\
	&\quad +\frac {\alpha_2}{ \Gamma(1-\alpha_2)}  t^{\alpha_1}  I_{0+}^{\alpha_1-\alpha_2} t^{\alpha_2-\alpha_1} \int_0^t \frac{t^{\alpha_1-\alpha_2}-s^{\alpha_1-\alpha_2}}{(t-s)^{\alpha_2+1}}  I_{0+}^{\alpha_2} s^{-\alpha_1} K_s^{n,1} ds.
	\end{align*}
	Hence, from the induction assumption we get
	\begin{align*}
	|K_t^{n+1,1}|&\leq \frac{GC^n}{\Gamma(n(\alpha_1-\alpha_2))} t^{\alpha_1} I_{0+}^{\alpha_1-\alpha_2} t^{n(\alpha_1-\alpha_2)-\alpha_1-\alpha_2} \\
	& \quad + \frac {\alpha_2}{ \Gamma(1-\alpha_2)}  \frac{GC^n}{\Gamma(n(\alpha_1-\alpha_2))} t^{\alpha_1}  I_{0+}^{\alpha_1-\alpha_2} t^{\alpha_2-\alpha_1} \int_0^t \frac{t^{\alpha_1-\alpha_2}-s^{\alpha_1-\alpha_2}}{(t-s)^{\alpha_2+1}}  I_{0+}^{\alpha_2} s^{-\alpha_1} s^{n(\alpha_1-\alpha_2)-\alpha_2} ds \\
	& \leq    \frac{GC^{n+1}  }{\Gamma((n+1)(\alpha_1-\alpha_2))} t^{(n+1)(\alpha_1-\alpha_2)-\alpha_2},
	\end{align*}
	for a suitable choice of the constant $C$.
This completes the proof of the lemma.
\end{proof}

\begin{lemma}\label{bound4}
	For every $n \in \N$ it holds
	\begin{equation*}
	|K_t^{n,2}| \leq \frac{ GC^n}{\Gamma(n(\alpha_1-\alpha_2))} t^{n(\alpha_1-\alpha_2)-\alpha_2},
	\end{equation*}
	where $G$ is a random variable satisfying \eqref{ecu8}  and $C$ is a positive finite constant.
\end{lemma}
\begin{proof}
	The proof is similar to the proof of Lemma \ref{bound3}.  We have
	\[
	K_t^{n,2} = t^{\alpha_2} \int_0^t g_s^n  \frac{t^{\alpha_1-\alpha_2}-s^{\alpha_1-\alpha_2}}{(t-s)^{\alpha_2+1} }ds,
	\]
	with
	\[
	g_s^n:= \left( I_{0+}^{\alpha_1}s^{\alpha_2-\alpha_1} D_{0+}^{\alpha_2} s^{\alpha_1-\alpha_2}\right)^n(s^{-\alpha_1} b_s), \quad n\in \N.
	\]
	By induction we will prove that
	\begin{equation}\label{ges}
	|g_s^n| \leq \frac{GC^n}{\Gamma(n(\alpha_1-\alpha_2))} s^{n(\alpha_1-\alpha_2)-\alpha_1}
	\end{equation}
	for every $n\in \N$. First we estimate $|g_s^n|$ for $n=1$. Using Lemma \ref{bound2} we obtain
\[
	|g_s^1| \leq G I_{0+}^{\alpha_1} s^{-\alpha_1-\alpha_2} =  G \frac{\Gamma(1-\alpha_1-\alpha_2)}{\Gamma(1-\alpha_2)} s^{-\alpha_2}.
\]
	Now assume that \eqref{ges} holds for some $n\in \N$. We have, in view of formula  \eqref{we}, 
	\begin{align*}
	|g_s^{n+1}|	& = \left| \left( I_{0+}^{\alpha_1}s^{\alpha_2-\alpha_1} D_{0+}^{\alpha_2} s^{\alpha_1-\alpha_2}\right) g_s^n \right|  \leq \left| I_{0+}^{\alpha_1-\alpha_2} g_s^n \right| \\
	& \quad + 
	\frac {\alpha_2} { \Gamma(1-\alpha_2)} 	 \left|I_{0+}^{\alpha_1} s^{\alpha_2-\alpha_1} \int_0^s \frac{s^{\alpha_1-\alpha_2} - u^{\alpha_1-\alpha_2}}{(s-u)^{\alpha_2+1}} g_u^n du \right| \\
	& \leq \frac{GC^n}{\Gamma(n(\alpha_1-\alpha_2))} I_{0+}^{\alpha_1-\alpha_2}s^{n(\alpha_1-\alpha_2)-\alpha_1} \\
	&\quad + \frac{GC^n \alpha_2}{ \Gamma(1-\alpha_2)\Gamma(n(\alpha_1-\alpha_2))}I_{0+}^{\alpha_1} s^{\alpha_2-\alpha_1} \int_0^s \frac{s^{\alpha_1-\alpha_2} - u^{\alpha_1-\alpha_2}}{(s-u)^{\alpha_2+1}} u^{n(\alpha_1-\alpha_2)-\alpha_1} du \\
	& \leq \frac{GC^n}{\Gamma(n(\alpha_1-\alpha_2))} I_{0+}^{\alpha_1-\alpha_2}s^{n(\alpha_1-\alpha_2)-\alpha_1} \\
	&\quad +\frac{GC^n c_1}{ \Gamma(1-\alpha_2)\Gamma(n(\alpha_1-\alpha_2))} I_{0+}^{\alpha_1} s^{ n(\alpha_1-\alpha_2) -\alpha_1-\alpha_2},
	\end{align*}
	where
	\[
	c_1= \int_0^1  \frac {1-x^{\alpha_1-\alpha_2}}{(1-x) ^{ \alpha_2+1}}dx.
	\]
	Taking into account that
	\[
	I_{0+}^{\alpha_1-\alpha_2}s^{n(\alpha_1-\alpha_2)-\alpha_1)}
	= \frac {\Gamma(n(\alpha_1-\alpha_2)-\alpha_1+1)} {\Gamma((n+1)(\alpha_1-\alpha_2)-\alpha_1+1}
	t^{ (n+1)(\alpha_1-\alpha_2)-\alpha_1}
	\]
	and
	\[
	I_{0+}^{\alpha_1}s^{n(\alpha_1-\alpha_2)-\alpha_1-\alpha_2}
	= \frac {\Gamma(n(\alpha_1-\alpha_2)-\alpha_1-\alpha_2+1)} {\Gamma((n+1)(\alpha_1-\alpha_2)-\alpha_2+1)}
	t^{ (n+1)(\alpha_1-\alpha_2)-\alpha_1},
	\]
	we deduce the existence of a constant $C$ such that
  \eqref{ges} holds for all $n\ge 1$. Overall, for $|K_2^n|$ we get the estimate
	\begin{align*}
	|K_2^n| &\leq \frac{GC^n}{\Gamma(n(\alpha_1-\alpha_2))} t^{\alpha_2} \int_0^t \frac{t^{\alpha_1-\alpha_2} - s^{\alpha_1-\alpha_2}}{(t-s)^{\alpha_2+1}} s^{n(\alpha_1-\alpha_2)-\alpha_1} ds 
	& = \frac{GC^n c_2}{\Gamma(n(\alpha_1-\alpha_2))}t^{n(\alpha_1-\alpha_2)-\alpha_2},
	\end{align*}
	where
	\[
	c_2= \int_0^1  \frac {1-x^{\alpha_1-\alpha_2}}{(1-x) ^{ \alpha_2+1}}  x^{-\alpha_1}dx.
	\]	
	This completes  the proof.
\end{proof}
Now we are in position to complete the proof of Novikov's condition. By \eqref{a1}, 
and Lemmas \ref{bound3} and \ref{bound4}, we can write
\[
|K_t^n| \leq  \frac{GC^n }{\Gamma(n(\alpha_1-\alpha_2))}  t^{n(\alpha_1-\alpha_2)-\alpha_2},
\]
where $G$ is a random variable satisfying \eqref{ecu8}  and $C$ is a positive finite constant.
Hence,  
\[
	|K_t| \leq  G t^{-\alpha_2}.
	\]
	and from \eqref{a2}, the same estimate holds for $\psi_t$.
Using this, the property of the random  variable $G$ and the fact that $-\alpha_2>-\frac12$ proves that
\[
\mathbb{E}\left[\exp \left(\lambda \int_0^t \psi_s^2 ds\right)\right] <\infty\qquad   \text{for all} \quad  \lambda>1,
\]
i.e. Novikov's condition is fulfillled. The proof of Theorem \ref{weak2} is complete. 
\end{proof}

\section{Uniqueness in law and pathwise uniqueness}
The results in this section are obtained by adapting the proofs of \cite[Theorem 4 and Theorem 5]{NO}. First we state the following.

\begin{theorem}\label{law}
	Suppose that one of the conditions in Theorem \ref{weak2} is satisfied. Then two weak solutions have the same distribution.
\end{theorem}

\begin{proof}
	We only give a proof in the case $H_1 < \frac12$ and $H_2>\frac12$. The idea in all of the remaining cases is essentially the same. Suppose that the triple $({B}^{H_1},{B}^{H_2},X)$ is a weak solution on a filtered probability space $(\Omega, \mathbb{F}, (\mathbb{F}_t)_{t \in [0, T]}, P)$ with underlying standard Brownian motion $W$. We define the two processes $u,v$ by
	\begin{equation}   \label{e1}
	u_t  = t^{-\alpha_1}\left[ 1+t^{\alpha_1+\alpha_2} I_{0+}^{\alpha_2} [t^{-\alpha_1-\alpha_2} I_{0+}^{\alpha_1}]\right] ^{-1}   \left( t^{\alpha_1} b(t, X_t) \right)
	\end{equation}
	and
	\begin{equation} \label{e2}
	v_t = t^{\alpha_2} I_{0+}^{\alpha_2} [t^{-\alpha_1-\alpha_2} I_{0+}^{\alpha_1} (t^{\alpha_1} u_t)]
	\end{equation}
	for every $t \in (0,T]$. Note that, as argued in the proof of Theorem \ref{weak2} we have
	\begin{equation} \label{up1}
	\left( K_{H_1}^{-1} \int_0^\bullet u_r dr \right) (t)=\left( K_{H_2}^{-1} \int_0^\bullet v_r dr \right) (t), \quad t \in (0,T].
	\end{equation}
	Also, from   \eqref{e1} and \eqref{e2}, it follows  that
	\begin{equation} \label{up2}
	u_t + v_t = b(t,X_t), \quad t \in (0,T].
	\end{equation}
	Define the process $(f_t)_{t\in (0,T]}$ with
	\[
	 f_t=\left( K_{H_1}^{-1} \int_0^\cdot u_r dr \right) (t), \quad t \in (0,T]
	 \]
	and the measure $Q$ by
	\[
	\frac{dQ}{dP} = \exp \left( \int_0^T f_s dW_s - \frac{1}{2} \int_0^T f_s^2 ds \right).
	\]
	Then arguing as on \cite[p.111]{NO}, we deduce that $\int_0^T f_s^2 ds < \infty$ almost surely,  and, moreover,  that
	\[
	 \|X\|_\infty \leq \left( \|B^{H_1}\|_\infty + \|B^{H_2}\|_\infty +cT \right) e^{cT}
	 \]
	and
	\[
	 |X_t - X_s| \leq | B_t^{H_1}- B_s^{H_1}| + | B_t^{H_2}- B_s^{H_2}| +c|t-s| (1+\|X\|_\infty), \quad s,t \in [0,T],  
	 \]
	for some constant $c \in (0, \infty)$. Hence, application of Novikov's theorem yields
	\[
	\mathbb{E} \left[\frac{dQ}{dP}\right] = 1,
	\]
	so that the process $(\wt{W}_t)_{t \in [0, T]}$ with
	\[
	 \wt{W}_t = W_t - \int_0^t f_s ds,\quad  t\in (0,T],
	 \]
	is a standard Brownian motion with respect to the probability measure $Q$. Note that by \eqref{up1} and \eqref{up2} we have
	\begin{equation}\label{up3}
	X_t = \int_0^t K_{H_1} (t,s) d \wt{W}_s + \int_0^t K_{H_2} (t,s) d \wt{W}_s, \quad t \in [0,T].
	\end{equation}
	Let
	\[
	\psi_t = t^{-\alpha_1}\left[ 1+t^{\alpha_1+\alpha_2} I_{0+}^{\alpha_2} [t^{-\alpha_1-\alpha_2} I_{0+}^{\alpha_1}]\right] ^{-1}   \left( t^{\alpha_1} b(t, B^{H_1}_t + B^{H_2}_t) \right) 
	\]
	and
	\[
	 g_t=\left( K_{H_1}^{-1} \int_0^\cdot \psi_r dr \right) (t)
	 \]
	for every $t \in (0,T]$. As a consequence from \eqref{up3} and the same calculations made on \cite[p.111]{NO} we obtain for every bounded measurable functional $\Psi$ on $C([0,T])$ that
	\begin{align*}
	\mathbb{E} [ \Psi (X)] = \mathbb{E} \left[ \Psi (B^{H_1} + B^{H_1}) \left( \exp \left( -\int_0^T g_s dW_s + \frac{3}{2} \int_0^T g_s^2 ds \right) \right) \right],
	\end{align*}
	where the expectation is taken with respect to the probability measure $P$. Since this equality holds for every weak solution, we obtain the uniqueness in law.
\end{proof}
The next result is obtained as a corollary of Theorem \ref{law}.

\begin{cor} \label{unique}
	Suppose that one of the conditions in Theorem \ref{weak2} is satisfied. Then two weak solutions defined on the same filtered probability space coincide almost surely.
\end{cor}

\section{Existence of strong solutions}
Now we turn to the proof of uniqueness and existence of strong solutions. By Corollary \ref{unique} it suffices to show the existence of strong solutions. If $\min(H_1,H_2) >\frac12$ the existence follows from the assumption that $b$ is (H\"older) continuous. The main contribution in this section is the proof of existence of strong solutions in case $\min(H_1,H_2) \leq \frac12$, more particularly the statement of Proposition \ref{kr} below establishing a Krylov-type estimate. Then, in view of the steps executed in the proof of \cite[Proposition 7 and Theorem 8]{NO}, we will have established our main result which reads as follows.

\begin{theorem}
	Let one of the conditions in Theorem \ref{weak2} be satisfied. Then there exists a unique strong solution.
\end{theorem}

\begin{proof}
	As described above the result follows from Proposition \ref{kr} below, which is the analogue of \cite[Proposition 6]{NO}.
\end{proof}

\begin{prop} \label{kr}
	Suppose that $\min(H_1,H_2) \leq \frac12$ and that the function $b$ is uniformly bounded. Let $X$ be a weak solution and $\rho > 1+\min(H_1,H_2)$. There exists a constant $C \in (0, \infty)$ only depending on $T$, $\|b\|_\infty$ and $\rho$ such that for every measurable function $g\colon [0,T] \times \rr \to [0, \infty)$ it holds
	\[
	\mathbb{E} \left[ \int_0^T g(t,X_t) dt \right] \leq C \left( \int_0^T \int_\rr g(t,y)^\rho dy dt \right)^{\frac{1}{\rho}} ,
	\]
	where the expectation is with respect to the underlying probability measure $P$.
\end{prop}

\begin{proof}
	Throughout this proof denote by $c \in (0, \infty)$ a generic constant only depending on $T$, $\|b\|_\infty$ and $\rho$. Without loss of generality assume that $H_1 = \min(H_1,H_2) \leq \frac12$. Define the processes $(u_t)_{t\in (0,T]}$, $(f_t)_{t\in (0,T]}$ and the probability measure $Q$ as explained in the proof of Theorem \ref{law}. Let $\alpha, \beta \in (1, \infty)$ be two constants with $\frac{1}{\alpha} + \frac{1}{\beta} = 1$. We denote by $\mathbb{E}_Q$ the expectation with respect to $Q$. Then by H\"older's inequality
	\begin{align} \label{kp1}
	\begin{split}
	\mathbb{E} \left[ \int_0^T g(t,X_t) dt \right] & \leq c \mathbb{E}_Q \left[ \left( \frac{dP}{dQ}\right) ^\alpha \right]^{\frac{1}{\alpha}} \left( 	\mathbb{E}_Q \left[ \int_0^T g(t,X_t)^\beta dt \right] \right)^{\frac{1}{\beta}}.
	\end{split}
	\end{align}
	Due to the proof of Theorem \ref{weak2} we obtain for every $\alpha>1$ that
	\begin{align} \label{kp2}
	\begin{split}
	\mathbb{E}_Q \left[ \left( \frac{dP}{dQ}\right) ^\alpha \right]^{\frac{1}{\alpha}} & = \mathbb{E}_Q \left[ \exp \left( -\alpha \int_0^T f_s dW_s + \frac{\alpha}{2} \int_0^T f_s^2 ds \right) \right] < \infty.
	\end{split}
	\end{align}
	Now we estimate the term
	\[
		\mathbb{E}_Q \left[ \int_0^T g(t,X_t)^\beta dt \right] 
		\]
	in the following. Since there exists a standard Brownian motion $\wt{W}$ with respect to $Q$ such that for every $t \in [0,T]$ we have
	\[
	X_t = \int_0^t K_{H_1} (t,s) d \wt{W}_s + \int_0^t K_{H_2} (t,s) d \wt{W}_s,
	\]
	under the probability measure $Q$, the random variable $X_t$ follows a centered normal distribution with variance
	\[
	\sigma^2(t) = \int_0^t \left( K_{H_1} (t,s) + K_{H_2} (t,s) \right) ^2 ds.
	\]
	Let $\delta \in (1+H_1, \infty)$ and $\gamma$ be the constant such that $\frac{1}{\gamma} + \frac{1}{\delta} =1$. Hence, by H\"older's inequality
	\begin{align*}
	\mathbb{E}_Q \left[ \int_0^T g(t,X_t)^\beta dt \right] & = \int_0^T \frac{1}{\sqrt{2\pi \sigma^2(t)}} \int_\rr g(t,y)^\beta e^{-\frac{y^2}{2\sigma^2(t)}} dy dt \\
	& \leq c  \left( \int_0^T \int_\rr g(t,y)^{\beta \delta}  dy dt \right)^{\frac{1}{\delta}} \left( \int_0^T \int_\rr \sigma(t)^{-\gamma}e^{-\frac{\gamma y^2}{2\sigma^2(t)}}   dy dt \right)^{\frac{1}{\gamma}} \\
	& = c  \left( \int_0^T \int_\rr g(t,y)^{\beta \delta}  dy dt \right)^{\frac{1}{\delta}} \left( \int_0^T  \sigma(t)^{1-\gamma} dt \right)^{\frac{1}{\gamma}} .
	\end{align*}
	Combining the latter inequality with \eqref{kp1} and \eqref{kp2} it suffices to show that the expression
	\[
	\int_0^T  \sigma(t)^{1-\gamma} dt 
	\]
	is finite in order to complete the proof. But indeed
	\begin{align*}
	\int_0^T  \sigma(t)^{1-\gamma} dt & \leq \int_0^T \left(  \int_0^t K^2_{H_1} (t,s)  ds \right) ^{\frac{1-\gamma}{2}} dt \\
	&   \leq c \int_0^T t^{H_1 (1-\gamma)} dt   <\infty.
	\end{align*}
\end{proof}
 
\bibliographystyle{amsplain}
\bibliography{lit}



 




\end{document}